\documentclass[A4paper]{amsart}
\input xypic
\usepackage{MnSymbol}
\usepackage{bbm}
\usepackage{amsmath}
\usepackage{graphicx}
\usepackage[final,hidelinks]{hyperref} 

\usepackage{amsthm}
\usepackage{mathrsfs}
\usepackage{amsopn}
\usepackage[all]{xy}
\usepackage{color}
\usepackage{subcaption}

\usepackage{tikz-cd}
\usepackage{ tipa }
\usepackage{longtable}

\usepackage[left=3.5cm,right=3.5cm,top=3.5cm,bottom=3.5cm]{geometry}

\usepackage{marginnote}
\usepackage{todonotes}

\usepackage{tikz}
\usetikzlibrary{patterns} 
\usepackage{wrapfig}
\usetikzlibrary{matrix,arrows,calc}

\usepackage{blkarray} 

\newtheorem{theorem}{Theorem}[section]
\newtheorem{proposition}[theorem]{Proposition}
\newtheorem{lemma}[theorem]{Lemma}
\newtheorem{corollary}[theorem]{Corollary}

\theoremstyle{definition}

\newtheorem{example}[theorem]{Example}

\newcounter{bean}

\newcommand{\Z}{\ensuremath{\mathbb{Z}}}
\newcommand{\C}{\ensuremath{\mathbb{C}}}

\newcommand{\zk}{\mathcal{Z}_K}

\newcommand{\K}{K}

\newcommand{\mP}{\mathcal{P}}
\newcommand{\mz}{\mathcal{Z}}

\newcommand{\dk}{\Delta_m^k}

\newcommand{\bbp}{\mathbb{P}}
\newcommand{\pw}{\partial^w}

\newcommand{\zm}{\mathcal{Z}_{m}}
\newcommand{\lk}{\mathrm{Link}}
\newcommand{\rt}{\mathrm{Rest}}
\newcommand{\st}{\mathrm{Star}}
\newcommand{\zl}{\mathcal{Z}_{L}}
\newcommand\scalemath[2]{\scalebox{#1}{\mbox{\ensuremath{\displaystyle #2}}}}
\renewcommand{\DJ}{\text{Davis-Januszkiewicz}}

\newcommand{\catk}{\mathrm{CAT}(K)}

\newcommand{\ima}{\mathrm{Im}\, }
\newcommand{\hp}{\text{homotopy commutative}}

\newcommand{\bx}{\mathbf{x}}

\newcommand{\colim}{\mathrm{colim}\, }
\newcommand{\hocolim}{\mathrm{hocolim}\, }

\newcommand{\catkd}{\text{$\catk$-diagram}}

\newcommand{\uhocolim}{\underset{\sigma\in K}\hocolim}

\title{Homotopy types of partial quotients for a certain case}

\author{Xin Fu} 
\address{department of mathematics, western university, 
london, ontario, n6a 5b7, canada 
} 
\email{xfu82@uwo.ca} 
\subjclass[2010]{55P15, 55R05}

\begin{document}

\begin{abstract}
In this paper, we determine the homotopy type of the quotient space $\mz_{\Delta^k_m}/S^1_d$, given by the moment-angle complex $\mz_{\Delta_m^k}$ under the diagonal circle action.
\end{abstract}

\maketitle


\section{Introduction}

For a simplicial complex $K$ on $[m]=\{1, 2,\ldots,m\}$, a {\it moment-angle complex} $\zk$ is defined by a union of product spaces, i.e., $\zk=\underset{\sigma\in K}\bigcup (D^2,S^1)^{\sigma}$, where $(D^2,S^1)^{\sigma}$ denotes $Y_1\times \ldots\times Y_m$ for which $Y_i=S^1$ if $i\notin \sigma$ and otherwise $Y_i=D^2$. Hence, by definition, a moment-angle complex has a natural coodinatewise $T^m$-action.  The {\it partial quotient}  is the quotient space $\zk/H$, where $H$ is a subtorus (a subgroup isomorphic to a torus). 

The cohomology of partial quotients $\zk/H$ is  identified with an appropriate Tor-algebra due to Panov~\cite{P}.  In addition, Franz~\cite{Franz} introduced the twisted product of Koszul complex whose cohomology algebraically isomorphic to $H^*(\zk/H)$  and also showed that the  cup product of  partial quotients differs with the standard multiplication on the Tor-algebra in general but they are isomorphic provided if $2$ is invertible in the coefficient.

Besides, the homotopy theoretical applications of moment-angle complexes are beautiful. Bahri-Bendersky-Cohen-Gitler~\cite{BBCG} showed that the suspension of a moment-angle complex splits into    a wedge of suspensions of the geometrical realisations of full subcomplexes. Porter~\cite{PORTER} and Grbi\'c-Theriault~\cite{GT2, GT1} proved that the homotopy type of moment-angle complexes for shifted complexes is a wedge of spheres. In particular, the $k$-skeleton $\Delta_m^k$ of an $(m-1)$-simplex is a simplicial complex consisting of all subsets of $[m]$ with cardinality at most $k+1$. It is a typical example of shifted complexes and there is a homotopy equivalence $\mz_{\Delta_m^k}\simeq 
\underset{j=k+2}{\overset{m}\bigvee}  \binom{m}{j} \binom {j-1} {k+1}S^{k+j+1}$ (see~\cite[Corollary 9.5]{GT2}). We adapt these ideas to  study  the partial quotient  $\mz_{\Delta_m^k}/S^1_d$  and prove the following statement.

\

{\noindent \bf Theorem~} For $0\leq k\leq m-2$, there is a homotopy equivalence
\[
\mz_{\Delta^k_{m}}/S_d^1\simeq 
\C P^{k+1}\vee\mz_{\Delta^k_{m-1}}
\vee
 (\underset{i=1}{\overset{k}\vee}S^{2i-1}*\mz_{\Delta_{m-i-1}^{k-i}})\vee
 (S^{2k+1}*T^{m-k-1}).
\]

Note that if $k=m-2$, then by definition, the quotient space $\mz_{\Delta^{m-2}_{m}}/S_d^1$ is $\C P^{m-1}$. The content of Section~\ref{S:pre} provides key lemmas  for proceeding the proof of the main result in Section~\ref{S:proof}.

\

{\it Acknowledgement.}~ I would like to express my  gratitude to Prof. Jelena Grbi\'{c} for her patient PhD supervision in many perspectives.
I also thanks to PhDs  for a lot of mathematical discussions in Southampton.



\section{Preliminaries}
\label{S:pre}

Let $K$ be a simplicial complex on $[m]$. We always assume that $\emptyset\in K$. Let $\catk$ be its face category whose objects are faces of $K$ and morphisms are face inclusions. A $\catkd$~$F$ of CW-complexes  is a functor from $\catk$ to $CW_*$, where $CW_*$ denotes the category of based, connected CW-complexes.

 We describe a construction of homotopy colimit for a $\catkd$ $F$, following a construction of the homotopy colimit in~\cite{BBCG, WZZ}   for a diagram $\mP\longrightarrow CW_*$, where $\mP$ is a poset (partially ordered set).
A $\catkd$ $F$ is equivalent to a diagram from a poset $\bar{K}$ to $CW_*$, where $\bar{K}$ denotes the poset associated to $K$ which has elements consisting of faces of $K$, ordered by the reverse inclusion.
 Then the construction $\underset{\sigma \in K}\hocolim F(\sigma)$  relies on the order complex $\Delta(\bar{K})$, which is $\mathrm{Cone}K^{\prime}$, the cone on the  barycentric subdivision of $K$.
 We adapt the construction  in~\cite[Section 4]{BBCG} of  homotopy colimit for a diagram $\mathcal{P}\longrightarrow CW_*$ to a $\catkd$ $F$, since objects and morphisms in $\catk$ form  a poset which is exactly $\bar{K}$.

Recall that $\mathrm{Cone}\, K^{\prime}$ has a vertex set $\{\sigma\in K\}$ including the empty face.
For $\sigma \in K$, denote by $X(\sigma)$  the full subcomplex of $\mathrm{Cone}\, K^{\prime}$ on the vertex set $\{\tau\in K\mid \sigma\subseteq \tau\}$. For faces $\sigma\subseteq \tau$ of $K$, then $X(\tau)$ is a subcomplex of $X(\sigma)$ and denote by $j_{\tau,\sigma}\colon X(\tau) \longrightarrow X(\sigma)$ the simplicial inclusion. Note that $X(\emptyset)= \mathrm{Cone}\, K^{\prime}$.
With a $\catkd~F$ and a subface $\sigma$ of $\tau$,  there are two types of  related maps $\alpha$ and $\beta$ defined by
\[
\begin{array}{lccc}
&\alpha=\mathrm{id} \times F(i_{\sigma,\tau}) \colon &X(\tau) \times F(\sigma)& \longrightarrow X(\tau) \times F(\tau)\\
&\beta  =j_{\tau,\sigma} \times \mathrm{id} \colon & X(\tau) \times F(\sigma)& \longrightarrow X(\sigma)
\times F(\sigma).
\end{array}
\]

Given a $\catkd~F$ of based CW complexes, the \emph{homotopy colimit} of $F$ is 
 a disjoint union  
$\underset{\sigma\in K}\coprod X(\sigma) \times F(\sigma)$ after identifications
\begin{equation}\label{eqhocolimd}
\underset{\sigma\in K}\hocolim F = (\underset{\sigma\in K}\coprod X(\sigma) \times F(\sigma))/\sim
\end{equation}
where $(\mathbf{x}, u)\sim (\mathbf{x}^{\prime}, u^{\prime})$ whenever $\alpha(\bx,u)=\beta(\bx^{\prime}, u^{\prime})$.

Let us denote $T^{\sigma}=\{(t_1, \ldots, t_m)\in T^m\mid t_j=1~\text{if}~j\notin \sigma\}$  is a $|\sigma|$-torus for  $\sigma\subseteq [m]$.
Thus the quotient group
 $T^m/T^{\sigma}=\{(t_1,\ldots, t_m)\in t^m \mid t_j=1~\text{if}~j\in \sigma\}$ is an $(m-|\sigma|)$-torus.
 For $\sigma\subseteq \tau\subseteq [m]$, there exists a quotient map $T^m/T^{\sigma}\longrightarrow T^m/T^{\tau}$ projecting $t_j$ to $1$ if $j\in \tau$ but $j\not\in \sigma$.
 This defines a $\catkd$ $D(\sigma)=T^m/T^{\sigma}$.
We show that the moment-angle complex provides a candidate for the homotopy colimit of the $\catkd$ $D(\sigma)$.
 \begin{example}[moment-angle complex]\label{hocolimzkd} 
Consider a $\catkd~D$ defined by
$D(\sigma)=T^m/T^{\sigma}$ with quotient maps 
$T^m/T^{\sigma}\longrightarrow T^m/T^{\tau}$ for $\sigma\subseteq \tau$ of $K$.
 We describe the homotopy colimit of $D$ by~\eqref{eqhocolimd}. 
First,  for every $\sigma\in K$, we have $X(\sigma)\times F(\sigma)\subseteq X(\emptyset)\times F(\emptyset)$.
We conclude that every element $(\bx, \mathbf{u})$ from $X(\sigma)\times F(\sigma)$ is equivalent to the same element $(\mathbf{x}, \mathbf{u})$ in $X(\emptyset)\times F(\emptyset)$ by considering the two types of maps $\alpha$ and $\beta$ corresponding to $\emptyset \subseteq \sigma$. 
Thus $\uhocolim D\simeq X(\emptyset)\times F(\emptyset)/\sim$.
To describe the equivalence relation on $X(\emptyset)\times F(\emptyset)$, we rely on the transitive property of an equivalence relation.
 That is to say,  $(\bx,\mathbf{u})\sim (\bx^{\prime}, \mathbf{u}^{\prime})$ in $X(\emptyset) \times F(\emptyset)$ if and only if there exists $\sigma\in K$ and an element $(\mathbf{y}, \mathbf{v})\in X(\sigma)\times F(\sigma)$ such that $(\bx,\mathbf{u})\sim (\mathbf{y},\mathbf{v})$ and $(\mathbf{y},\mathbf{v})\sim (\bx^{\prime}, \mathbf{u}^{\prime})$.
In this way, we have $\mathbf{x}=\mathbf{y}=\bx^{\prime}$ and  $\mathbf{u}_j=\mathbf{u}^{\prime}_j$ for $j\not \in \sigma$, where $u_j$ and $u_j^{\prime}$ are the $j$-th coordinate of $\mathbf{u}$ and $\mathbf{u}^{\prime}$ respectively. Note that $u_j=u^{\prime}_j$ for $j\not \in \sigma$ if and only if $\mathbf{u}^{-1}\mathbf{u}^{\prime}\in T^{\sigma}$. 
 Then, we have
\[
\uhocolim D\simeq \mathrm{Cone} K^{\prime} \times T^m/\sim
\]
where $(\mathbf{x}, \mathbf{u})\sim (\mathbf{y}, \mathbf{u}^{\prime})$ if and only if  for some $\sigma\in K$, $\mathbf{x}=\mathbf{y} \in X(\sigma)$ and $\mathbf{u}^{-1}\mathbf{u}^{\prime}\in T^{\sigma}$. Note that the space $ \mathrm{Cone} K^{\prime} \times T^m/\sim$ is $T^m$-equivariantly homeomorphic to $\zk$, where $T^m$ acts on the second coordinate (see~\cite{BP2, DJ}).
\end{example}

An analogy to this is that if $H\cap T^{\sigma}$ is trivial, then the partial quotient $\zk/H$ is a candidate of homotopy colimit for a $\catkd$ $E$ by $E(\sigma)=T^m/(T^{\sigma}\times H)$ and quotient maps.

\subsection{Fibration sequences}
We apply Puppe’s theorem~\cite{Puppe}   to get homotopy fibrations. Our exposition below follows a description due to~\cite[p.180]{Far}.

 Let $\mathcal{E}$ be a $\catk$-diagram of spaces and let $B$ be a fixed space.  By a map  $f\colon \mathcal{E}\longrightarrow B$  bewteen  $\mathcal{E}$ and $B$, we mean that $f$ is a natural transformation from $\mathcal{E}$ to $\mathrm{Top}$ with a constant evaluation $f(\sigma)=B$ for every $\sigma\in \mathcal{E}$.
 With a map $f$ from $\mathcal{E}$ to a fixed space $B$, there exists an associated  diagram of fibres by taking the objectwise homotopy fibre. To be precise,
 a $\catkd$ $\mathrm{Fib}_f$ of fibres is defined by taking $\mathrm{Fib}_f(\sigma)$ to be the homotopy fibre of $f_{\sigma}\colon \mathcal{E}(\sigma)\longrightarrow B$
and morphisms $\mathrm{Fib}_f(\sigma)\longrightarrow \mathrm{Fib}_f(\tau)$ to be the corresponding maps between fibres induced by the map $\mathcal{E}(\sigma)\longrightarrow \mathcal{E}(\tau)$ for $\sigma\subseteq \tau$ in $K$.

Given a map $f\colon  \mathcal{E} \longrightarrow B$, there are two topological spaces associated.  One is the homotopy fibre of an induced map $\bar{f}\colon \uhocolim\mathcal{E}(\sigma) \longrightarrow B$ and another one is $\uhocolim \mathrm{Fib}_f(\sigma)$, the homotopy colimit of the $\catkd$ of fibres induced by $f$. Puppe's theorem states when these two spaces have the same homotopy type.

\begin{theorem}[\cite{Far, Puppe}]\label{puppethm}
 Let $\mathcal{E}$ be a $\catk$-diagram of spaces, let $B$ be a fixed connected space and let  $f\colon \mathcal{E}\longrightarrow B$ be any map bewteen  $\mathcal{E}$ and $B$. Assume that for $\sigma\subseteq \tau$ in $\catk$, the following diagram is commutative
\[\begin{tikzcd}
\mathcal{E}(\sigma) \arrow[r]\arrow[d] & \mathcal{E}(\tau)\arrow[d]\\
B\arrow[r, equal] & B.
\end{tikzcd}  
\]
Then the homotopy fibre of the induced map $\bar{f}\colon \uhocolim \mathcal{E}(\sigma)\longrightarrow B$ is homotopy equivalent to the homotopy colimit of a $\catk$-diagram $\mathrm{Fib_f}$ of fibres.
\end{theorem}

Puppe's theorem indicates the following lemma.

\begin{lemma}\label{fibrations}
Let $H$ be a subtorus of $T^{m}$ of rank $r$ satisfying $H\cap T^{\sigma}=\{\mathbf{1}\}$ for every $\sigma\in K$. Then the quotient map $\zk\overset{q}\longrightarrow \zk/H$ makes the following diagram of homotopy fibrations commutative up to homotopy
\[
\begin{tikzcd}
\zk \arrow[r]\arrow[d, "q"] & DJ_K \arrow[r, "j"] \arrow[d, equal] & BT^m \arrow[d, "B\pi"]\\
\zk/H \arrow[r] & DJ_K \arrow[r, "(B\pi)\circ j"] & B(T^m/H)
\end{tikzcd}
\]
where $j$ is a canonical inclusion.
%
%
\end{lemma}

\begin{proof}
If $H\cap T^{\sigma}$ is trivial for every $\sigma\in K$, then we have a diagram of fibrations
\begin{equation}\label{fibrebun}
\begin{tikzcd}
T^m/T^{\sigma} \arrow[r]\arrow[d] & BT^{\sigma} \arrow[r] \arrow[d, equal]  & BT^m \arrow[d, "B\pi"]\\
T^m/(T^{\sigma}\times H) \arrow[r]  & BT^{\sigma} \arrow[r]  & B(T^m/H).
\end{tikzcd}
\end{equation}
Consider the $\DJ$ space as $DJ_K=(BS^1,*)^K\simeq \underset{\sigma\in K}{\mathrm{hocolim}} BT^{\sigma}$. 
The inclusion $j_{\sigma}\colon BT^{\sigma} \longrightarrow BT^m$ and its composition with the quotient map $\pi j_{\sigma}\colon BT^{\sigma}\longrightarrow B(T^m/H)$, give two maps from a $\catkd~ DJ$ (by sending $\sigma\in K$ to $(BS^1, *)^{\sigma}$) $BT^m$ and $B(T^m/H)$, respectively.  By the fibre bundles (\ref{fibrebun}), the $\catk$-diagrams $D$ with $D(\sigma)=T^m/T^{\sigma}$ and morphisms are projections, and $E$  with $E(\sigma)=T^m/(T^{\sigma}\times H)$ and morphisms are projections, are the induced  $\catk$-diagrams of fibres for $(BS^1, *)^K \overset{j}\longrightarrow BT^m$ and $(BS^1, *)^K\overset{(B\pi)\circ i}\longrightarrow B(T^m/H)$, respectively. Objectwise, the quotient map $D(\sigma)\longrightarrow E(\sigma)$ is the induced map between fibres.

Note that these two maps $j$ and $(B\pi) \circ j$ satisfy the condition in Puppe's theorem. 
A direct consequence of Puppe's theorem is that 
$\underset{\sigma \in K}\hocolim D(\sigma)$ and $\underset{\sigma \in K}\hocolim E(\sigma)$ are the homotopy fibres of maps $DJ_K \overset{j}\longrightarrow BT^m$ and $DJ_K \overset{(B\pi)\circ j}\longrightarrow B(T^m/H)$, respectively.  
  According to the construction~\eqref{eqhocolimd} of  the homotopy colimit, the objectwise quotient map $D(\sigma)\longrightarrow E(\sigma)$ will induce a quotient map between $X(\emptyset)\times D(\sigma)/\sim$ and $X(\emptyset)\times E(\sigma)/\sim$.
   These candidates~\eqref{eqhocolimd} of the homotopy colimit of $D$ and $E$  are homeomorphic to $\zk$ and $\zk/H$.  When we replace $X(\emptyset)\times D(\sigma)/\sim$ and $X(\emptyset)\times E(\sigma)/\sim$ by $\zk$ and $\zk/H$ due to the homeomorphism, the quotient map between $X(\emptyset)\times D(\sigma)/\sim$ and $X(\emptyset)\times E(\sigma)/\sim$ induces the quotient map between $\zk$ and $\zk/H$,  since  $X(\emptyset)\times D(\sigma)/\sim$ and $\zk$ are $H$-equivariantly homeomorphic.
\end{proof}

\begin{remark}
It can be shown that if $K$ does not have ghost vertices, then these two fibration sequences in Lemma~\ref{fibrations} splits after loop because of the existence of sections in both cases.  The  long exact sequence of homotopy groups associated to $\zk/H\longrightarrow DJ_K\longrightarrow B(T^m/H)$ implies that $\zk/H$ is simply-connected.
The condition that $H\cap T^{\sigma}$ is trivial for every $\sigma\in K$  is equivalent to that $H$ acts freely on $\zk$.
\end{remark}

\subsection{Homotopy pushouts of fibres.}
Here we rely on Mather's Cube Lemma~\cite{Mather} to obtain a homotopy pushout among fibres.

\begin{lemma}[Cube Lemma \cite{Mather, Martin}]\label{Mather}
Consider a cube diagram whose faces are homotopy commutative.
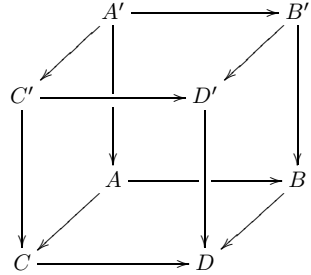
\[
\scalemath{0.85}{\xymatrix{
& A^{\prime} \ar[ld] \ar[rr] \ar'[d][dd]  & & B^{\prime}  \ar[dd]\ar[ld]
\\
C^{\prime} \ar[rr]\ar[dd]
& & D^{\prime} \ar[dd]
\\
&  A \ar[ld] \ar'[r][rr]
& & B \ar[ld]
\\
C\ar[rr]
& & D
}}
 \]
If the bottom square $A-B-C-D$ is a homotopy pushout and all four sided square are homotopy pullbacks, then the top square $A^{\prime}-B^{\prime}-C^{\prime}-D^{\prime}$ is also a homotopy pushout.
\end{lemma}

Given a map $D\longrightarrow Z$, there is a commutative diagram 
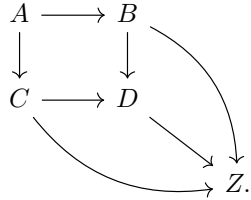
\[
\begin{tikzcd}
A \arrow[r]\arrow[d] &B \arrow[d] \arrow[ddr, bend left] &\\
C\arrow[r]\arrow[drr, bend right] & D \arrow[dr]& \\
 & & Z.
\end{tikzcd}
\]

A  special case of cube lemma observes that the top square $A^{\prime}-B^{\prime}-C^{\prime}-D^{\prime}$ is obtained by taking the homotopy fibre, respectively, through mapping each $A, B, C, D$ into a fixed space $Z$ given a map $D\longrightarrow Z$. So that, if $A-B-C-D$ is a homotopy pushout, then the square of fibres on the top $A^{\prime}-B^{\prime}-C^{\prime}-D^{\prime}$ is a homotopy pushout too.

In particular, a pushout $K_1\longleftarrow K_1\cap K_2\longrightarrow K_2$ of simplicial complexes gives rise to  a  pushout $(BS^1,*)^{\overline{K}_2}\longleftarrow (BS^1,*)^{\overline{K_1\cap K_2}}\longrightarrow (BS^1,*)^{\overline{K}_1}$ of Davis-Januszkiewicz  spaces, where $(BS^1, *)^{\overline{K}}$ denotes the polyhedral product allowing the ghost vertices, considering the corresponding simplicial complex as a subcomplex of $K_1\cup K_2$.
Since $(BS^1, *)$ is a pair of CW complexes, the maps between $\DJ$ spaces induced by simplicial inclusions   are cofibrations. So  this pushout in terms of $\DJ$ spaces is also a homotopy pushout.
 Mapping $(BS^1, *)^K$ to $BT^m$ and $B(T^m/H)$ as in Lemma~\ref{fibrations}, we  have the homotopy fibres $\zk$ and $\zk/H$. Hence by Lemma~\ref{Mather}, there are two homotopy pushouts in terms of moment-angle complexes $\zk$ and their quotients $\zk/H$ and the maps among them are induced by simplicial inclusions.
 
 If $K_1$ is a subcomplex of $K$,  denote by $\mz_{\overline{K}_1}$ the moment-angle complex  allowing ghost vertices on the vertex set of $K$.
 For two based spaces $X$ and $Y$, the {\it half-smash product} is $X\ltimes Y\simeq X\times Y/X\times *$ and the {\it join} is $X*Y\simeq  \Sigma X\wedge Y$.
 Under the assumption of Lemma~\ref{fibrations}, the next statement follows.
 
 \begin{lemma} \label{cubeq}
 Let $K=K_1\cup K_2$ on $[m]$. Suppose that $H$ is a subtorus of $T^m$ such that $H\cap T^{\sigma}=\{\mathbf{1}\}$ for any $\sigma\in K$.  There is a commutative cube diagram
  \[
\scalemath{0.75}{\xymatrix{
& \mz_{\overline{K_1\cap K_2}} \ar[ld] \ar[rr] \ar'[d][dd] 
 & & \mz_{\overline{K}_2}  \ar[dd]\ar[ld]
\\
\mz_{\overline{K}_1} \ar[rr]\ar[dd]
& & \zk \ar[dd]
\\
&  \mz_{\overline{K_1\cap K_2}}/H \ar[ld] \ar'[r][rr]
& & \mz_{\overline{K}_2}/H
 \ar[ld]
\\
\mz_{\overline{K}_1}/H
\ar[rr]
& & \zk/H
}}
 \]
where the top and bottom are homotopy pushouts, whose maps are induced by simplicial inclusions and all vertical maps are quotient maps. 
 \end{lemma}

\begin{proof}
 The is a consequence of Cube Lemma and Lemma~\ref{fibrations}.
\end{proof}

\begin{example} \label{pre}
Let $K$ be the following simplicial complex with $K_1$ and $K_2$ pictured below. Consider the diagonal $S^1$-action on $\zk$.
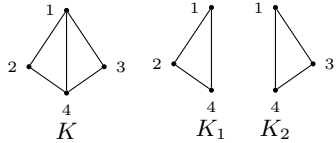
\begin{center}
\begin{tikzpicture}[scale=0.5]
\tikzstyle{point}=[circle,draw=black,fill=black,inner sep=0pt,minimum width=1.5pt,minimum height=0pt] 
\node (5)[label=below: \small{$K$}] at(0, -1.5){};
\node (1)[point, label=left:\tiny{$1$}] at (0,1) {}; 
\node (2)[point, label=left:\tiny{$2$}] at (-1,-0.5) {}; 
\node (3)[point, label=right:\tiny{$3$}] at (1,-0.5) {};
\node (4)[point, label=below:\tiny{$4$}] at (0,-1.2) {};
\draw (1) -- (2) -- (4) -- (3) -- (1)--(4);
\end{tikzpicture}~~~~~~
 \begin{tikzpicture}[scale=0.5]
\tikzstyle{point}=[circle,draw=black,fill=black,inner sep=0pt,minimum width=1.5pt,minimum height=0pt] 
\node (5)[label=below: \small{$K_1$}] at(0, -1.5){};
\node (1)[point, label=left:\tiny{$1$}] at (0,1) {}; 
\node (2)[point, label=left:\tiny{$2$}] at (-1,-0.5) {}; 
\node (4)[point, label=below:\tiny{$4$}] at (0,-1.2) {};
\draw (1) -- (2) -- (4) --(1);
\end{tikzpicture}~~~~~~
\begin{tikzpicture}[scale=0.5]
\tikzstyle{point}=[circle,draw=black,fill=black,inner sep=0pt,minimum width=1.5pt,minimum height=0pt] 
\node (5)[label=below: \small{$K_2$}] at(0, -1.5){};
\node (1)[point, label=left:\tiny{$1$}] at (0,1) {}; 
\node (3)[point, label=right:\tiny{$3$}] at (1,-0.5) {};
\node (4)[point, label=below:\tiny{$4$}] at (0,-1.2) {};
\draw (1)  -- (4) -- (3) -- (1);
\end{tikzpicture}
\end{center}
In this case, we have the following  spaces (up to homotopy)
\[
\scalemath{1}{
\mz_{\overline{K_1\cap K_2}}\simeq S^1\times S^1,~\mz_{\overline{K_1\cap K_2}}/S^1_d\simeq S^1,~
\mz_{\overline{K}_i}\simeq S^1\times S^5,~\mz_{\overline{K}_i}/S^1_d\simeq S^5, i=1,2.
}
\]

The diagram in Lemma~\ref{cubeq} indicates a homotopy commutative diagram by a replacement of spaces due to homotopy equivalences
 \[
\scalemath{0.7}{\xymatrix{
& S^1\times S^1 \ar[ld]_{*\times \mathrm{id}} \ar[rr]^{\mathrm{id}\times *} \ar'[d][dd] 
 & & S^1\times S^5  \ar[dd]\ar[ld]
\\
S^5\times S^1\ar[rr]\ar[dd]
& & \zk \ar[dd]
\\
& S^1 \ar[ld]_{*} \ar'[r][rr]^{*}
& & S^5
 \ar[ld]
\\
S^5
\ar[rr]
& & \zk/S^1_d
}}
 \]
 where the top and bottom square are homotopy pushout. Since the fundamental group $\pi_1(S^5)$ is trivial, the homotopy types of $\zk$ and $\zk/S^1_d$ are 
 \[
 \zk\simeq S^1*S^1\vee (S^1\ltimes S^5) \vee (S^5\rtimes S^1)~\text{and}~\zk/S^1_d\simeq S^2\vee 2S^5.
 \]
\end{example}
We  continue to consider the homotopy types of $\mz_{\Delta^k_m}/S^1_d$ by taking a pushout of simplicial complexes in the next section.

\section{Homotopy types of partial quotients}
\label{S:proof}

In this section, we study homotopy types of  $\zk/S^1$. In particular, we determine the homotopy type of the quotient space $\mz_{\Delta^k_m}/S^1_d$ under the diagonal action.
We first consider  properties of moment-angle complexes under subtorus actions in the next lemma.
\begin{lemma}\label{quotientlemma} 
Let $K$ be a simplicial complex on $[m]$ and  let $H$ be a subtorus of $T^m$ acting on $\zk$ and $r=\mathrm{rank} H$.

$\mathrm{(a)}$ For $\sigma\in K$,  $(D^2, S^1)^{\sigma}$ is an $H$-invariant subspace of $\zk$. Consequently, for any simplicial subcomplex $L\subseteq K$, $\zl$ is an $H$-subspace of $\zk$. 

$\mathrm{(b)}$ Let $\Phi\colon H\times \zk \longrightarrow \zk$ be the action map. Then there exists a homeomorphism $\mathrm{sh}\colon H\times \zk \longrightarrow H\times \zk$ such that $p_2\circ \mathrm{sh}=\Phi$, where $p_2$ is a projection  $H\times \zk\longrightarrow\zk$. 

$\mathrm{(c)}$ The action map $\Phi\colon H\times \zk \longrightarrow \zk$ induces a map $\bar{\Phi}\colon H\ltimes \zk \longrightarrow \zk$ with a homotopy cofibre $C_{\bar{\Phi}}\simeq H*\zk$.
\end{lemma}
\begin{proof}
(a) Since $H$ is a subtorus of $T^m$, there is an isomorphism $T^r\cong H<T^m$ given by a choice of basis and an $m\times r$ integral matrix $S=(s_{ij})$ such that  $g=(g_1,\ldots, g_m)\in H$ has the form $g_i=t^{s_{i1}}_1\ldots t^{s_{ir}}_r$~$\text{with}~ (t_1, \ldots, t_r)\in T^r$.
Let  $\mathbf{z}=(z_1,\ldots, z_m)\in (D^2, S^1)^{\sigma}$, that is, $z_i\in D^2$ if $i\in \sigma$ and $z_i\in S^1$ if $i\notin \sigma$. Recall that $S^1$ acts on $D^2$ by a rotation. Thus if $z_i\in \mathrm{Int} D^2$, then $g_i\cdot z_i\in \mathrm{Int} D^2$ and if $z_i\in \partial D^2$, then $g_i\cdot z_i\in \partial D^2$. Therefore, $g_i\cdot z_i\in D^2$ if $i\in \sigma$, otherwise $g_i\cdot z_i\in S^1$. Thus $g\cdot \mathbf{z}=(g_1\cdot z_1, \ldots, g_m\cdot z_m) \in (D^2, S^1)^{\sigma}$.

(b) Define the shearing map $H\times \zk \overset{\mathrm{sh}}\longrightarrow H\times \zk$ by $\mathrm{sh}(g, \mathbf{z})=(g, \Phi(g, \mathbf{z}))$ for $g\in H$ and $\mathbf{z}\in \zk$. It is a homeomorphism with  inverse $\mathrm{sh}^{-1}(g,\mathbf{z})=(g, g^{-1}\mathbf{z})$. Thus  $p_2\circ \mathrm{sh}=\Phi$.

(c)  Let $*$ be the base point $(1,  \ldots, 1)$ of $\zk$. Since the image  $\Phi|_{H\times *}$ is in $T^m$ and the inclusion $T^m\longrightarrow \zk$ is null homotopic, thus $\Phi|_{H\times *}$ is also null homotopic.  The homotopy cofibration $H\hookrightarrow H\times \zk \longrightarrow H\ltimes \zk$ gives an induced map $\bar{\Phi}\colon H\ltimes \zk \longrightarrow \zk$ with $\bar{\Phi}\circ q\simeq \Phi$.
Note that $H*\zk$ is the homotopy pushout of $H\overset{p_1}\longleftarrow H\times \zk \overset{p_2}\longrightarrow \zk$. By the second statement, the shearing map $\mathrm{sh}$ is a homeomorphism and $\Phi=p_2\circ sh$, $H*\zk$ is the homotopy pushout of $H\overset{p_1}\longleftarrow H\times \zk \overset{\Phi}\longrightarrow \zk$. Pinching out $H$, we have $C_{\bar{\Phi}}\simeq H*\zk$.
\end{proof}

\subsection{Free circle actions}
Now we focus on circle actions on $\zk$. Suppose that  $S^1\cong H=\{(t^{s_1}, \ldots, t^{s_m})\mid t\in S^1\}$ is a circle subgroup $T^m$, where $s_i\in \Z$.
Let $\Lambda$ be the associated integral matrix of the projection $T^m\longrightarrow T^m/H$.  The relation between $S$ and $\Lambda$ is  as follows.  Since $H$ is a circle subgroup of $T^m$, there exists an integral $m\times (m-1)$-matrix $S^{\prime}$ such that the $m\times m$-matrix $\begin{pmatrix}
S\mid S^{\prime}
\end{pmatrix}$ is invertible, where $S=(s_1,\ldots, s_m)$.   Then
$\begin{pmatrix}
\Lambda^{\prime}\\
\Lambda
\end{pmatrix}$
is the inverse matrix of 
$\begin{pmatrix}
S\mid S^{\prime}
\end{pmatrix}$
where $\Lambda^{\prime}=(\lambda_{ij}^{\prime})$ is an integral $(1\times m)$-vector and $\Lambda=(\lambda_{ij})$ is the  integral $(m-1)\times m$-matrix representing the quotient map $T^m\longrightarrow T^m/H$.
Following this, if $s_1=\pm 1$, then the matrix $\begin{pmatrix}
s_1&\mathbf{0}\\
\mathbf{s} & I_{m-1}
\end{pmatrix}$ has an invertible matrix $\begin{pmatrix}
s_1&\mathbf{0}\\
-s_1\mathbf{s} & I_{m-1}
\end{pmatrix}$, where $\mathbf{s}=(s_2,\ldots, s_m)$. Thus $\Lambda=\begin{pmatrix}
-s_1\mathbf{s} & I_{m-1}
\end{pmatrix}$ such that $\mathrm{Ker}\Lambda=H$.

The next statement applies to the special case of quotient spaces $\zk/S^1$ under free circle actions when $K$ has ghost vertices. 
\begin{lemma}\label{gost}
Suppose that $\{v\}$ is a ghost vertex of $K$. Let $S^1$ acts on $\zk$ by $(s_1,\ldots, s_m)$. If $s_{v}=\pm 1$, then $S^1$ acts on $\zk$ freely and $\zk/S^1 \simeq  \zl$, where $L=K_{\bar{V}}$ is the full subcomplex of $K$ on $\bar{V}=V(K)\setminus \{v\}$.
\end{lemma}
\begin{proof}
Without loss of generality, we can assume $\{1\}$ is a ghost vertex of $K$. Then $\zk=S^1\times \mz_{L}$, where $\mz_{L}$ is an $S^1$-space by $(s_2, \ldots, s_m)$. If $s_1=\pm 1$, then $S^1$-action on $\zk$ is an $S^1$-diagonal action  on the product space $S^1\times \mz_{L}$.
Let $\Phi, \Phi^{-1}$ be maps $S^1\times \mz_{L}\longrightarrow \mz_{L}$ where $\Phi$ is the group action and $\Phi^{-1}(g,\mathbf{z})=(g^{-1}, \mathbf{z})$.
Then if $s_1=1$, $\Phi^{-1}$ will induce an $S^1$-equivariant homeomorphism  $\zk/S^1=S^1\times_{S^1} \mz_{L}\cong \mz_{L}$, whose inverse  is given by sending $\mathbf{z}\in \mz_{L}$ to $[(1, \mathbf{z})]\in \zk/S^1$. Similarly, if $s_1=-1$, then the action map $\Phi$ will induce an $S^1$-equivariant homeomorphism.
\end{proof}

For a simplicial complex $K$ and $v\in V(K)$, let
\[\begin{split}
\mathrm{Link}_K(v)&=\{\sigma\in K\mid (v)*\sigma\in K, v\notin \sigma\}\\
\mathrm{Star}_K(v)&=\{\sigma\in K\mid (v)*\sigma\in K\}=(v)*\mathrm{Link}_K(v)\\
\mathrm{Rest}_K(v) &=\{\sigma\in K\mid V(\sigma)\subseteq V(K)\setminus \{v\}\}.
\end{split}
\]

There exists a pushout of simplicial complexes 
\[
\begin{tikzcd}
\mathrm{Link}_K(v)\arrow[d] \arrow[r] &\mathrm{Rest}_K(v) \arrow[d]\\
 \mathrm{Star}_K(v) \arrow[r] & K
\end{tikzcd}
\]
which induces a topological pushout of corresponding Davis-Januszkiewicz spaces.
 Mapping these spaces to $B(T^m/S^1)$, denote by $F_{\lk}, F_{\st}$ and $F_{\rt}$ the correspond homotopy fibres, respectively. Then there is a diagram of homotopy pushouts as follows. 

\[
\scalemath{0.8}{\xymatrix{
& F_{\mathrm{Link}}
\ar[ld] \ar[rr] \ar'[d][dd] 
 & & F_{\mathrm{Rest}}  \ar[dd]\ar[ld]
\\
F_{\mathrm{Star}} \ar[rr]\ar[dd]
& & \zk/S^1 \ar[dd]
\\
&  (BS^1, *)^{\mathrm{Link}_K(v)} \ar[ld] \ar'[r][rr]
& & (BS^1, *)^{\mathrm{Rest}_K(v)}
 \ar[ld]
\\
(BS^1, *)^{\mathrm{Star}_K(v)}
\ar[rr]
& & (BS^1, *)^{K}
}}
 \]

If a circle action on $\zk$ satisfies the condition in Lemma~\ref{gost}, it is possible to identify the homotopy types of these fibres for special cases.

\begin{theorem}\label{circlet}
Let $S^1$ acts on $\zk$ freely. Assume that there exits a vertex  $v\in K$ such that $s_{v}=\pm 1$. 

$\mathrm{(a)}$ There exist homotopy equivalences
\[
F_{\mathrm{Link}}\simeq \mathcal{Z}_{\mathrm{Link}_K(v)},~ F_{\mathrm{Rest}}\simeq \mathcal{Z}_{\mathrm{Rest}_K(v)},~ F_{\mathrm{Star}}\simeq  \mathcal{Z}_{\mathrm{Link}_K(v)}/S^1.
\]

$\mathrm{(b)}$ The quotient space $\zk/S^1$ is the homotopy pushout of the diagram \[
\mz_{\lk_K(v)}/S^1\overset{q}\longleftarrow\mz_{\lk_K(v)}\overset{\iota}\longrightarrow \mz_{\rt_K(v)}\]
where $\iota$ is the map induced by the simplicial inclusion $\lk_K(v)\longrightarrow \rt_K(v)$ and $q$ is the quotient map.
\end{theorem}
\begin{proof}
(a) Without loss of generality, assume that $v=\{1\}$. Since $\mathrm{Link}_K(1)$ and $\mathrm{Rest}_K(1)$ are on the vertex set $\{2, \ldots, m\}$, 
 $F_{\mathrm{Link}}\simeq \mathcal{Z}_{\mathrm{Link}_K(1)},~ F_{\mathrm{Rest}}\simeq \mathcal{Z}_{\mathrm{Rest}_K(v)}$ by Lemma~\ref{gost}.
Since $S^1$ acts on $\mz_{\st_K(1)}$ freely, its quotient $\mz_{\st_K(1)}/S^1$ is homotopy equivalent to the Borel construction $ES^1\times_{S^1}\mz_{\st_K(1)}$,
 where $\mz_{\st_K(1)}=D^2\times \mz_{\lk_K(1)}$.  Since $S^1$ acts on $\mz_{\lk_K(1)}$ freely, $F_{\st}=\mathcal{Z}_{\st_K(1)}/S^1 \simeq ES^1\times_{S^1}\mz_{\st_K(1)}\simeq ES^1 \times_{S^1} (D^2\times \mathcal{Z}_{\lk_K(1)})\simeq \mz_{\lk_K(1)}/S^1$. 

(b) It suffices to identify the maps between these fibres. Since $s_1=\pm 1$, the matrix $\Lambda$ representing the projection $T^m\longrightarrow T^m/S^1$ is given by
 $\begin{pmatrix}
-s_1\mathbf{s} & I_{m-1}
\end{pmatrix}$ with $\mathbf{s}=(s_2, \ldots, s_m)^{t}$. Therefore, the composite $BT^{m-1}\overset{Bj}\longrightarrow BT^m \longrightarrow B(T^m/S^1)$ is the identity map, where $j$ is an inclusion of $T^{m-1}$ to the last $m-1$ coordinates of $T^m$. 
Thus
for $L$ being $\lk_K(1)$ or $\rt_K(1)$, the composite $(BS^1, *)^L\longrightarrow BT^m \overset{B\Lambda}\longrightarrow B(T^m/S^1)$ is the standard inclusion $(BS^1, *)^L\longrightarrow BT^{m-1}$ if $T^{m-1}$ is identified with $T^m/S^1$.
 Therefore, the map between the fibres $F_{\lk}\longrightarrow F_{\rt}$ is the inclusion between the corresponding moment-angle complexes $\mz_{\lk}\overset{\iota}\longrightarrow \mz_{\rt}$.

There exists an induced free circle action on $\mz_{\lk_K(1)}$ given by $g\cdot (z_2,\ldots, z_m)= (g^{s_2}\cdot z_2,\ldots, g^{s_m}\cdot z_m)$. 
We first note that $\mathrm{Im}\, \mathbf{s}=\{(t_2^{s_2}, \ldots, t_m^{s_m})\mid (t_2,\ldots, t_m)\in T^{m-1}\}$ is  a  circle subgroup of $T^{m-1}$. Because  we assume that $\{1\}\in K$,  the freeness condition of  a circle action on $\zk$ implies that $\gcd(s_2, \ldots, s_m)=1$. 
To see that this induced action is free, send $(z_2, \ldots, z_m) \in \mz_{\lk_K(1)}$ to $(0, z_2, \ldots, z_m)\in \mz_{\st_K(1)}$. The isotropy group of $(0, z_2, \ldots, z_m)$ under the original $S^1$-action  by $(s_1, s_2, \ldots, s_m)$ is equal to the isotropy group of $(z_2, \ldots, z_m)$ under the induced $S^1$-action  by $(s_2, \ldots, s_m)$. Since the original $S^1$-action  acts on $\mz_{\st_K(1)}$ freely, the isotropy group of $(0, z_2, \ldots, z_m)$ is trivial, which means that the $S^1$-action on $\mz_{\lk_K(1)}$ by $(s_2, \ldots, s_m)$ is free.

This circle subgroup of $T^{m-1}$ has an associated integral matrix $\pi$ representing the quotient map $T^{m-1}\longrightarrow T^{m-1}/S^1$.  There is a homotopy commutative diagram of fibrations
\begin{equation}\label{fibre1}
\begin{tikzcd}
&\mz_{\lk_K(1)}\arrow[r]\arrow[d, equal]& (BS^1, *)^{\lk_K(1)}\arrow[d, equal] \arrow[rr, "(B\Lambda)\circ i"] & & B(T^m/S^1)\arrow[d, "\simeq"] \\
&\mz_{\lk_K(1)}\arrow[r]\arrow[d, "q"]& (BS^1, *)^{\lk_K(1)}\arrow[d, equal] \arrow[rr, "\eta"] & & BT^{m-1}\arrow[d, "B\pi"] \\
&\mz_{\lk_K(1)}/S^1 \arrow[r] \arrow[d, equal] &(BS^1, *)^{\lk_K(1)} \arrow[rr, "\gamma=(B\pi)\circ \eta"] \arrow[d, hook, "j_2"] & & B(T^{m-1}/S^1)\arrow[d, hook, "j_2"] \\
& \mz_{\lk_K(1)}/S^1 \arrow[r] &BS^1\times (BS^1, *)^{\lk_K(1)} \arrow[rr, "\mathrm{id}\times \gamma"] & & BS^1\times B(T^{m-1}/S^1)
\end{tikzcd}
\end{equation}
where the top rectangle is obtained by $T^m/S^1$ being identified with $T^{m-1}$, the second rectangle is due  to Lemma~\ref{fibrations}, and $q$ is a quotient map and $j_2$ is an inclusion into the second coordinate.

In fact, the  homotopy fibration at the bottom row in~\eqref{fibre1} is equivalent to  the homotopy fibration obtained by mapping $(BS^1, *)^{\st_K(1)}$ to $B(T^m/S^1)$
\[
F_{\st}\longrightarrow (BS^1, *)^{\st_K(1)}\overset{(B\Lambda)\circ i}\longrightarrow B(T^m/S^1).
\]
The relation between $(s_2, \ldots, s_m)$ and $\pi$ implies that $T^m/S^1$ is isomorphic to $\mathrm{Im}\, \mathbf{s} \times \mathrm{Im}\, \pi$, where $\mathrm{Im}\, \mathbf{s}$ and $\mathrm{Im}\, \pi$ are torus groups with rank $1$ and $m-2$, respectively.
Thus there are   isomorphisms 
 $T^m/S^1\overset{M_1}\longrightarrow  \ima\mathbf{s} \times \ima \pi \overset{M_2}\longrightarrow S^1\times T^{m-2}$, which are represented by an $(m-1)\times (m-1)$-integral invertible matrices $M_1$ and $M_2$.
Let $M=M_2 M_1$. Composing $BM$ with $(B\Lambda)\circ i$, we  have a  diagram of homotopy fibrations
\begin{equation}\label{fibre2}
\begin{tikzcd}
F_{\st}\arrow[r] \arrow[d] &(BS^1, *)^{\st_K(1)}\arrow[rr, "(B\Lambda)\circ i"]\arrow[d, equal] & &B(T^m/S^1) \arrow[d, "BM"] \\
F \arrow[r] &(BS^1, *)^{\st_K(1)}\arrow[rr, "(BM)\circ (B\Lambda)\circ i"] & & BS^1\times BT^{m-2}
\end{tikzcd}
\end{equation} where the left square is $\hp$ and the right one is commutative and all vertical maps are homotopy equivalences. 

Since $(BS^1, *)^{\st_K(1)}=BS^1\times (BS^1, *)^{\lk_K(1)}$, the composite $(BM)\circ B\Lambda\circ i=\mathrm{id}\times \gamma$.
Combining these  homotopy commutative diagrams~(\ref{fibre1}) and (\ref{fibre2}), the simplicial inclusion $\lk_K(1)\longrightarrow \st_K(1)$ induces a quotient map of the fibres $\mz_{\lk_K(1)}\overset{q}\longrightarrow \mz_{\lk_K(1)}/S^1$.
\end{proof}
In some special case,  the map $\mz_{\lk_K(v)}\longrightarrow \mz_{\rt_K(v)}$ is null homotopic. For example, if for some $v\in K$ such that $\lk_K(v)=\emptyset$, then $\mz_{\lk_K(v)}\longrightarrow \mz_{\rt_K(v)}$ is null homotopic (\cite[Lemma 3.3]{GT1}). If so, there is a homotopy splitting of the quotient $\zk/S^1$.

\begin{corollary}
Let $K$ and $S^1$ satisfy  the  assumption in Theorem~\ref{circlet}. Suppose that for the same vertex $v$,   the map $\mz_{\lk_K(v)}\longrightarrow \mz_{\rt_K(v)}$ is null homotopic. Then 
there exists a homotopy splitting 
$
\zk/S^1\simeq \mz_{\rt_K(v)}\vee C_q
$, where $C_q$ is the homotopy cofibre of the quotient map $\mz_{\lk_K(v)}\overset{q}\longrightarrow \mz_{\lk_K(v)}/S^1$.

In particular, if $\lk_K(v)=\emptyset$, then $\zk/S^1\simeq \mz_{\rt_K(v)}\vee S^2\vee(S^1*T^{m-2})$.
\end{corollary}
\begin{proof}
If the map $\mz_{\lk_K(v)}\longrightarrow \mz_{\rt_K(v)}$ is null homotopic,  there is an iterated homotopy pushout
\[
\begin{tikzcd}
\mz_{\lk_K(v)} \arrow[r] \arrow[d, "q"] & *\arrow[d] \arrow[r] &\mz_{\rt_K(v)}\arrow[d]\\
\mz_{\lk_K(v)}/S^1 \arrow[r] & C_q \arrow[r] &\zk/S^1.
\end{tikzcd}
\] Thus the first statement follows.

If $\lk_K(v)=\emptyset$, then $\mz_{\lk_K(v)}\simeq T^{m-1}$ and  $\st_K(v)=\{v\}$.
Consider the following diagram of fibration sequences
\[
\begin{tikzcd}
\mz_{\emptyset} \arrow[r] \arrow[d, "q"] & * \arrow[d, hook] \arrow[r] &B(T^m/S^1)\arrow[d, equal]\\
\mz_{\emptyset}/S^1 \arrow[r] \arrow[d, "p"', "\simeq"] & BS_v^1 \arrow[r] \arrow[d, equal]& B(T^m/S^1)\arrow[d, "\simeq"]\\
\Omega BT^{m-2} \arrow[r] & BS_v^1\arrow[r, "\mathrm{id}\times *"] & BS^1\times BT^{m-2}.
\end{tikzcd}
\] 
Here the top diagram between homotopy fibrations is induced  by $\emptyset\longrightarrow \{v\}$ and the bottom diagram is an equivalence of fibration sequences, proved as a special case of diagram~(\ref{fibre2}) in Theorem~\ref{circlet}, due to the isomorphism  $T^m/S^1\cong S^1\times T^{m-2}$. Since $p$ is a homotopy equivalence, we have $C_q\simeq C_{p q}$. Note that the composition $p q$ is induced by projecting $T^m/S^1\longrightarrow T^{m-2}$. Precisely, it is the map $T^m/S^1\overset{\cong}\longrightarrow S^1\times T^{m-2} \overset{p_2}\longrightarrow T^{m-2}$.
Therefore, it remains to identify the homotopy cofibre of $p_2$. 

Let $\pi_2\colon X\times Y \longrightarrow Y$ be a projection, where $X$ and $Y$ are two connected CW-complexes. Consider the following homotopy commutative diagram
\[\begin{tikzcd}
X \times Y \arrow[r, "\pi_1"] \arrow[d, "\pi_2"] & X\arrow[d] \arrow[r] &*\arrow[d]\\
Y\arrow[r] &X*Y\arrow[r] &C_{\pi_2}
\end{tikzcd}
\]
where the left and right diagrams are homotopy pushouts. Since $X\longrightarrow X*Y$ is null homotopic, $C_{\pi_2}\simeq \Sigma X\vee X*Y$.
Thus $C_{q}\simeq \Sigma S^1\vee (S^1* T^{m-2})$. 
\end{proof}
\begin{example}
Denote by $\zm$ the moment-angle complex corresponding to $m$ disjoint points. If $S^1$ acts freely  on $\zm$ by $(s_1, \ldots, s_m)$ with some $s_j=\pm 1$, then $\zm/S^1\simeq \mz_{m-1}\vee S^2\vee (S^1*T^{m-2})$.
\end{example}

\subsection{Homotopy types of cofibres}\label{zkms1d}
In this section, we determine homotopy cofibre $C_{k,m}$ of  the quotient map $q_{k, m}\colon \mathcal{Z}_{\Delta_m^k}\longrightarrow \mathcal{Z}_{\Delta_m^k}/S^1_d$ under the diagonal action. 
Note that if $K=\Delta^{k}_{m}$ is on  the vertex set $\{1, \ldots, m\}$, then $\mathrm{Link}_K\{1\}$ is simplicially isomorphic to $\Delta^{k-1}_{m-1}$ on the vertex set $\{2,\ldots, m\}$. Thus we have a pushout of simplicial complexes 
\[
\begin{tikzcd}
\Delta^{k-1}_{m-1}\arrow[d]\arrow[r]& \Delta^{k}_{m-1}\arrow[d]\\
(1)*\Delta^{k-1}_{m-1}\arrow[r] & \Delta^{k}_{m}.
\end{tikzcd}
\]
This pushout implies  homotopy pushouts of the corresponding moment-angle complexes and of their quotient spaces under the diagonal action by Lemma~\ref{cubeq}.
\begin{equation}\label{tpoutk}
\begin{tikzcd}
&S^1\times \mathcal{Z}_{\Delta^{k-1}_{m-1}}\arrow[r, "\mathrm{id}\times *"]\arrow[d, "*\times \mathrm{id}"] & S^1\times \mathcal{Z}_{\Delta^{k}_{m-1}}\arrow[d, "f_{k, m}"]\\
&\mathcal{Z}_{\Delta^{k-1}_{m-1}} \arrow[r] & \mz_{\Delta^{k}_{m}}
\end{tikzcd}
~~~
\begin{tikzcd}
& \mathcal{Z}_{\Delta^{k-1}_{m-1}}\arrow[r, "\simeq *"]\arrow[d, "q_{k-1, m-1}"] &  \mathcal{Z}_{\Delta^{k}_{m-1}}\arrow[d, "g_{k, m}"]\\
&\mathcal{Z}_{\Delta^{k-1}_{m-1}}/S_d^1 \arrow[r] & \mz_{\Delta^{k}_{m}}/S_d^1
\end{tikzcd}
\end{equation}
where $f_{k, m}$ is a  map induced by the simplicial inclusion $\Delta^{k}_{m-1} \longrightarrow \Delta^{k}_{m}$ and the map $g_{k, m}$  is induced by $f_{k,m}$ between the quotient spaces.

The left diagram in (\ref{tpoutk}) implies an iterated homotopy pushout
\[
\begin{tikzcd}
&S^1\times \mathcal{Z}_{\Delta^{k-1}_{m-1}}\arrow[r]\arrow[d, "*\times \mathrm{id}"] 
& S^1\arrow[d] \arrow[r] 
& S^1 \times \mathcal{Z}_{\Delta^{k}_{m-1}} 
\arrow[d, "f_{k, m}"]\\
&\mathcal{Z}_{\Delta^{k-1}_{m-1}} 
\arrow[r]
 & S^1 *  \mathcal{Z}_{\Delta^{k-1}_{m-1}}
  \arrow[r] 
 & \mz_{\Delta^{k}_{m}}
\end{tikzcd}
\]
which induces the following iterated homotopy pushout  after pinching out $S^1$
\begin{equation}\label{pushout}
\begin{tikzcd}
&S^1\ltimes \mathcal{Z}_{\Delta^{k-1}_{m-1}}
\arrow[r]\arrow[d, "*\ltimes \mathrm{id}"]
& *\arrow[r] \arrow[d]
& S^1\ltimes \mathcal{Z}_{\Delta^{k}_{m-1}}
\arrow[d, "\bar{f}_{k, m}"]\\
&\mathcal{Z}_{\Delta^{k-1}_{m-1}}
 \arrow[r]& 
 S^1*  \mathcal{Z}_{\Delta^{k-1}_{m-1}} 
 \arrow[r, "h_{k,m}"] & \mz_{\Delta^{k}_{m}}.
\end{tikzcd}
\end{equation}
Thus the right square of (\ref{pushout}) implies  a  splitting homotopy cofibration  $S^1\ltimes \mathcal{Z}_{\Delta^{k}_{m-1}} \overset{\bar{f}_{k,m}}\longrightarrow \mz_{\Delta^{k}_{m}}\longrightarrow C_{\bar{f}_{k, m}}$, 
 where the homotopy cofibre  $C_{\bar{f}_{k, m}}$ is homotopic to $ S^1*\mathcal{Z}_{\Delta^{k-1}_{m-1}}$.

Since the map $\mz_{\Delta_{m-1}^{k-1}}\longrightarrow \mz_{\Delta_{m-1}^k}$ is null homotopic,
 the right homotopy pushout in~(\ref{tpoutk}) also implies  an iterated homotopy pushout
\begin{equation}\label{pushout2}
\begin{tikzcd}
\mz_{\Delta_{m-1}^{k-1}}\arrow[r] \arrow[d, "q_{k-1,m-1}"] & * \arrow[r]\arrow[d] & \mz_{\Delta_{m-1}^k}\arrow[d, "g_{k,m}"]\\
\mz_{\Delta_{m-1}^{k-1}}/S_d^1 \arrow[r] & C_{k-1, m-1} \arrow[r, "h_{k,m}^{\prime}"] & \mz_{\Delta_{m}^k}/S^1_d.
\end{tikzcd}
\end{equation}
 The right square of (\ref{pushout2}) implies a splitting homotopy cofibration $\mz_{\Delta_{m-1}^k}\overset{g_{k,m}}\longrightarrow \mz_{\Delta_m^k}/S^1_d\longrightarrow C_{g_{k,m}}$, where  the homotopy cofibre $C_{g_{k,m}}$ is homotopic to $C_{k-1, m-1}$.

\begin{lemma}\label{lemzkm}
There exists a homotopy equivalence $\mz_{\dk}/S^1_d\simeq \mz_{\Delta^k_{m-1}}\vee C_{k-1, m-1}$, where $C_{k-1, m-1}$ is the homotopy cofibre of the quotient map $\mz_{\Delta_{m-1}^{k-1}}\longrightarrow \mz_{\Delta_{m-1}^{k-1}}/S^1_d$.
\end{lemma}

 Hence, to determine the homotopy type of $\mz_{\dk}/S^1_d$, it suffices to determine the homotopy type of $C_{k,m}$.

\begin{lemma}
There exists a homotopy commutative diagram
\begin{equation}\label{cd}
\begin{tikzcd}
&S^1\ltimes \mathcal{Z}_{\Delta^{k}_{m-1}} \arrow[r, "\bar{\Phi}^{-1}"] \arrow[d, "\bar{f}_{k, m}"] &   \mathcal{Z}_{\Delta^{k}_{m-1}} \arrow[d, "g_{k, m}"]\\
 &\mathcal{Z}_{\Delta^{k}_{m}} \arrow[r, "q_{k, m}"] &  \mathcal{Z}_{\Delta^{k}_{m}}/S_d^1
\end{tikzcd}
\end{equation}
where $\bar{\Phi}^{-1}$ is induced by the map  $S^1\times \mathcal{Z}_{\Delta^{k}_{m-1}}\overset{\Phi^{-1}}\longrightarrow \mathcal{Z}_{\Delta^{k}_{m-1}}$ given by $\Phi^{-1}(t, \mathbf{z})=t^{-1}\cdot\mathbf{z}$.
\end{lemma}
\begin{proof}
The simplicial inclusion $\Delta_{m-1}^k\longrightarrow \Delta_m^k$ gives rise to a commutative diagram
\[
\begin{tikzcd}
S^1\times \mz_{\Delta_{m-1}^k} \arrow[r, "\alpha"] \arrow[d, "f_{k,m}"]& S^1\times_{S_d^1} \mz_{\Delta_{m-1}^k}\arrow[d, "\beta"]\\
\mz_{\Delta_{m}^k}\arrow[r, "q_{k,m}"]&\mz_{\Delta_{m}^k}/S_d^1
\end{tikzcd}
\]
where the horizontal maps $\alpha$ and $q_{k,m}$ are quotient maps and $\beta$ is a map between quotient spaces induced by $f_{k,m}$.
By Lemma~\ref{gost}, 
there is a homotopy equivalence
\[
S^1\times_{S_d^1}  \mz_{\Delta_{m-1}^k}\overset{\eta}\simeq \mz_{\Delta_{m-1}^k}
\]
where  $\eta$ sends $[(t, \mathbf{z})]$ to $\Phi^{-1}(t, \mathbf{z})$. 
It follows easily  that $\eta\circ \alpha (t,\mathbf{z})=t^{-1}\cdot \mathbf{z}=\Phi^{-1}(t, \mathbf{z})$. Thus,
 replacing  $S^1\times_{S^1} \mz_{\Delta_{m-1}^k}$ by its homotopy equivalent space $\mz_{\Delta_{m-1}^k}$ due to  ${\eta}$,
there is a homotopy commutative diagram,
\[
\begin{tikzcd}
S^1\times \mz_{\Delta_{m-1}^k} \arrow[r, "\Phi^{-1}"] \arrow[d, "f_{k,m}"]& \mz_{\Delta_{m-1}^k}\arrow[d, "\beta\circ \eta"]\\
\mz_{\Delta_{m}^k}\arrow[r, "q_{k,m}"]&\mz_{\Delta_{m}^k}/S_d^1
\end{tikzcd}
\] where $\beta\circ \eta$ coincides the map $g_{k,m}$ in the diagram~(\ref{tpoutk}), since they are the maps induced by $f_{k,m}$ after we have chosen an certain homotopy type of quotient spaces.

Since the restriction of $\Phi^{-1}$ to the first coordinate $S^1$ is null homotopic, we obtain the homotopy commutative diagram in the statement.
\end{proof}
The homotopy commutative diagram~(\ref{cd}) gives rise to the following homotopy commutative diagram 
\begin{equation}\label{keycd}
\begin{tikzcd}
&S^1\ltimes \mathcal{Z}_{\Delta^{k}_{m-1}}
 \arrow[r, "\bar{\Phi}^{-1}"]\arrow[d, "\bar{f}_{k,m}"] 
 &\mathcal{Z}_{\Delta^{k}_{m-1}}
  \arrow[r] \arrow[d, "g_{k,m}"] 
  &S^1*\mathcal{Z}_{\Delta^{k}_{m-1}} 
  \arrow[d] \\
& \mathcal{Z}_{\Delta^{k}_{m}} 
\arrow[r, "q_{k,m}"] \arrow[d, "r_{k,m}"] 
& \mathcal{Z}_{\Delta^{k}_{m}} /S_d^1 
\arrow[r]\arrow[d, "r^{\prime}_{k,m}"] 
& C_{k, m}
\arrow[d] \\
&C_{\bar{f}_{k,m}}
 \arrow[r, "\phi_{k, m}"] 
 & C_{g_{k, m}} 
 \arrow[r]
  & Q_{k, m}
\end{tikzcd}
\end{equation}
where each row and column is a homotopy cofibration and the first row is due to Lemma~\ref{quotientlemma}(c).
The homotopy pushouts~(\ref{pushout}) and (\ref{pushout2}) imply that $C_{\bar{f}_{k,m}}\simeq S^1*\mz_{\Delta_{m-1}^{k-1}}$ and $C_{g_{k,m}}\simeq C_{k-1,m-1}$ and the first and second columns of (\ref{keycd}) are splitting homotopy cofibrations. 

 We will determine the homotopy type of $C_{k,m}$. The idea is to find simplicial complexes $L^{k}_{j,m}$ such that their quotient spaces under diagonal actions give the homotopy type of the cofibre of the quotient map.
We firstly identify the homotopy type of maps $\phi_{k,m}$.

 \begin{lemma} \label{phiid}
 Let $K=\Delta^k_m$ and $L_{1,m}^k=K\cup \Delta_{\{1, 2, \ldots, m-1\}}$. Then $C_{\bar{f}_{k,m}}\simeq \mz_{L_{1,m}^k}$
  and $\mz_{L_{1,m}^k}/ S_d^1\simeq C_{g_{k,m}}$. 
  Under these homotopy equivalences,
  the map $\phi_{k,m}$ can be taken the quotient map $\mz_{L_{1,m}^k}\longrightarrow \mz_{L_{1,m}^k}/S^1_d$.
 \end{lemma}
\begin{proof}
Since $K\cap \Delta_{\{1, 2, \ldots, m-1\}}=\Delta^k_{m-1}$, we have a pushout of simplicial complexes
\[\begin{tikzcd}
\Delta^k_{m-1}\arrow[r]\arrow[d] & \Delta_{\{1, 2, \ldots, m-1\}}\arrow[d]\\
\Delta^k_{m}\arrow[r]&L_{1,m}^k.
\end{tikzcd}
\]
There are two homotopy pushouts of topological spaces, one of moment-angle complexes and one of quotient spaces of moment-angle complexes
\[
\begin{tikzcd}
\mz_{\Delta^k_{m-1}} \times S^1 
\arrow[r, "* \times \mathrm{id}"]
\arrow[d, "f_{k,m}"'] & S^1\arrow[d]\\
\mz_{\Delta^k_{m}}\arrow[r]& \mz_{L_{1,m}^k}
\end{tikzcd}
\text{and}
\begin{tikzcd}
\mz_{\Delta^k_{m-1}}
\arrow[r]\arrow[d, "g_{k,m}"]
 & *\arrow[d]\\
\mz_{\Delta^k_{m}}/S_d^1\arrow[r]& \mz_{L_{1,m}^k}/S_d^1.
\end{tikzcd}
\]
Pinching out $S^1$ in the left pushout above, we have a homotopy cofibration
\begin{equation}\label{cofibration}
\mz_{\Delta^k_{m-1}}\rtimes S^1 \overset{\bar{f}_{k,m}}\longrightarrow \mz_{\Delta^k_{m}} \longrightarrow \mz_{L_{1,m}^k}.
\end{equation}
Taking the corresponding quotient spaces of (\ref{cofibration}) and the homotopy commutative diagram~(\ref{keycd}), there exists a homotopy commutative diagram of homotopy fibrations
 \[\begin{tikzcd}
\mz_{\Delta^k_{m-1}}\rtimes  S^1 \arrow[r, "\bar{f}_{k,m}"]\arrow[d] 
&\mz_{\Delta^k_{m}} \arrow[r] \arrow[d, "q_{k,m}"]
 & \mz_{L_{1,m}^k}\arrow[d]\\
 \mz_{\Delta^k_{m-1}} \arrow[r, "g_{k,m}"] 
 & \mz_{\Delta^k_{m}}/S_d^1 \arrow[r] 
 & \mz_{L_{1,m}^k}/S_d^1. 
 \end{tikzcd} 
 \]
 Thus the maps $\phi_{k,m}$ in (\ref{keycd}) are quotient maps up to homotopy and $C_{\bar{f}_{k,m}}\simeq \mz_{L_{1,m}^k}$
  and $\mz_{L_{1,m}^k}/ S_d^1\simeq C_{g_{k,m}} \simeq C_{k-1, m-1}$. 
\end{proof}

 We have identified the homotopy type of $C_{k-1, m-1}$ as $\mz_{L_{1,m}^k}/ S_d^1$.
 We will continue to show that the homotopy cofibre $C_{k,m}$ has the following form.
\begin{theorem}\label{ckm}
There exists a homotopy equivalence
\[
C_{k,m} \simeq
 \C P^{k+2}\vee (\underset{i=1}{\overset{k+1}\vee}S^{2i-1}*\mz_{\Delta_{m-i}^{k+1-i}})\vee (S^{2k+3}*T^{m-k-2}).
\]
\end{theorem}
The main idea of the proof of Theorem~\ref{ckm} is to construct a sequence of simplicial complexes $L_{j,m}^k$ and iterate to determine the homotopy types of their quotient spaces under the diagonal action. 
We  give an explicit construction of these simplicial complexes $L_{j,m}^k$ from the $k$-skeleton $\dk$.

Denote by $\Delta_{\{i_1,\ldots, i_p\}}$ a simplex on vertices $\{i_1, \ldots, i_p\}$. Let $L^k_{0,m}=\dk$. Define $L_{1,m}^k=\dk \cup \Delta_{\{1,2,\ldots,m-1\}}$ and  $L_{j,m}^k=L_{j-1,m}^k\cup \Delta_{\{1, \ldots, \widehat{m-j+1}, \ldots, m\}}$, where $\widehat{m-j+1}$ means that this vertex is omitted.

We first prove that the simplicial inclusion $L_{j,m}^{k-1}\longrightarrow L_{j,m}^k$ induces a null homotopic map on corresponding moment-angle complexes.
\begin{lemma}\label{nullJ}
For $1\leq j\leq k+1$, the inclusion $J\colon \mz_{L_{j,m}^{k-1}}\longrightarrow \mz_{L_{j,m}^k}$ is null homotopic.
\end{lemma}
\begin{proof}
Let $K=\underset{m-j+1 \leq q\leq m}\bigcup \Delta_{\{1,\ldots,\hat{q}, \ldots, m-1\}}$. Thus $\mz_{K}=(\underset{m-j}\prod D^2) \times \mz_{\partial \Delta^{j-1}}$, where $\partial \Delta^{j-1}$ is the boundary of a simplex on vertices $\{m-j+1,\ldots, m\}$. Note that $L_{j,m}^{k}=\Delta_{m}^{k} \cup K$ and $\mz_{L_{j,m}^{k}}=\mz_{\Delta^{k}_m}\cup \mz_{K}$.

First, there is a filtration of simplicial complexes $\Delta^{k-1}_m\subseteq (1)*\Delta_{m-1}^{k-1}\subseteq \Delta_{m}^k$, where $\Delta^{k-1}_{m-1}$ in the middle is on vertices $\{2,\ldots, m\}$, which implies a filtration of simplicial complexes
 $L_{j,m}^{k-1}\subseteq ((1)*\Delta_{m-1}^{k-1})\cup K \subseteq L_{j, m}^k$.
In particular, 
$((1)*\Delta_{m-1}^{k-1})\cup K =(1)*(\Delta_{m-1}^{k-1}\cup K_1)$,
 where $K_1$ is the full subcomplex of $K$ on vertices $\{2, \ldots, m\}$.
Thus, the inclusion $J$ factors through the  corresponding moment-angle complexes 
\[\mz_{L_{j,m}^{k-1}} \overset{i_1}\longrightarrow D^2\times (\mz_{\Delta_{m-1}^{k-1}}\cup \mz_{K_1}) \overset{i_1^{\prime}}\longrightarrow \mz_{L_{j, m}^k}.\]
By the construction of $L^{k}_{j, m}$, $\Delta_{m-1}^{k-1}\cup K_1=L_{j, m-1}^{k-1}$ which is a full subcomplex of $L^{k-1}_{j, m}$ on vertices $\{2, \ldots, m\}$.
Denote by $r_1$ the retraction $\mz_{L_{j,m}^{k-1}} \longrightarrow \mz_{L_{j, m-1}^{k-1}}$. Then the map $i_1$ factors through $r_1$ and a coordinate inclusion $\iota_1\colon \mz_{L_{j, m-1}^{k-1}} \longrightarrow D^2\times \mz_{L_{j, m-1}^{k-1}}$ up to homotopy. Namely, there exists a  diagram
\[
\begin{tikzcd}
& \mz_{L_{j,m}^{k-1}} \arrow[d , "i_1"] \arrow[dl, "r_1"'] \arrow[dr, "J"]& \\
\mz_{L_{j, m-1}^{k-1}}\arrow[r, "\iota_1"'] &  D^2\times \mz_{L_{j, m-1}^{k-1}} \arrow[r, "i_1^{\prime}"'] & \mz_{L_{j, m}^k} 
\end{tikzcd}
\]
where the left triangle is homotopy commutative and the right one is commutative.
In particular, the composition $i_1^{\prime}\iota_1$ coincides with the map induced by the simplicial inclusion $L_{j,m-1}^{k-1}\longrightarrow L_{j,m}^k$ which has a filtration $L_{j,m-1}^{k-1} \overset{j_2}\longrightarrow L^{k}_{j, m-1}\overset{j_2^{\prime}}\longrightarrow L^k_{j, m}$. 

 The same strategy applies for $L_{j,m-1}^{k-1} \overset{j_2}\longrightarrow L^{k}_{j, m-1}$.
Repeating the above procedure, there are diagrams for $1\leq q\leq m-k-1$
\begin{equation}\label{repeat}
\begin{tikzcd}
& \mz_{L_{j,m-q+1}^{k-1}} \arrow[d , "i_q"] \arrow[dl, "r_q"'] \arrow[dr, "j_q"]& \\
\mz_{L_{j, m-q}^{k-1}}\arrow[r, "\iota_q"'] \arrow[dr, "j_{q+1}"'] &  D^2\times \mz_{L_{j, m-q}^{k-1}} \arrow[r, "i_q^{\prime}"'] & \mz_{L_{j, m-q+1}^k} \\
& \mz_{L_{j,m-q}^k} \arrow[ur, "j_{q+1}^{\prime}"'] &
\end{tikzcd}
\end{equation}
where each $L_{j, m-q}^{k-1}$ is a full subcomplex of $L_{j, m-q+1}^{k-1}$ on vertices $\{q+1, \ldots, m\}$, the top left triangle is homotopy commutative and the other two are commutative.

If $q=m-k-1$, observe the composition $\mz_{L_{j, k+1}^{k-1}}\overset{j_{m-k}}\longrightarrow \mz_{L_{j, k+1}^{k}} \overset{j_{m-k}^{\prime}}\longrightarrow \mz_{L_{j, k+2}^{k}}$. Since  $L_{j, k+1}^{k}$ is a full subcomplex of $L_{j, k+2}^{k}$ on vertices $\{m-k, \ldots, m\}$, it contains all subsets of $\{m-k, \ldots, m\}$ with cardinality at most $k+1$. Thus  $L_{j, k+1}^{k}$ is a simplex, which means that $j_{m-k}$ is null homotopic. Chasing the homotopy commutative diagram~(\ref{repeat}), $j_{m-k}$ is a factor of $J$ up to homotopy. Hence, $J$ is null homotopic.
\end{proof}

\begin{proposition}\label{hmpjkm}
There exist  homotopy equivalences
$\mz_{L_{j,m}^k}\simeq  S^1* \mz_{L_{j-1,m-1}^{k-1}}$
and $\mz_{L_{j,m}^k}/S^1_d\simeq C_{q^{k-1}_{j-1, m-1}}$,
where $C_{q^{k}_{j, m}}$ denotes the homotopy cofibre of the    quotient map $\mz_{L_{j,m}^k} \overset{q_{j,m}^k}\longrightarrow \mz_{L_{j,m}^k}/S^1_d$.
Consequently, we have the homotopy types of the following spaces
 \[
\mz_{L_{j,m}^k}\simeq \begin{cases}
 S^{2j-1}* \mz_{\Delta_{m-j}^{k-j}}~&\text{if}~1\leq j\leq k+1\\
 S^{2k+3}~&\text{if}~j=k+2\end{cases}\]
and
\[
\mz_{L_{j,m}^k}/S_d^1\simeq
 \begin{cases}
 \C P^{k+1} \vee (\underset{i=j}{\overset{k}\vee}S^{2i-1}*\mz_{\Delta_{m-i-1}^{k-i}})\vee (S^{2k+1}*T^{m-k-2})~&\text{if}~1\leq j\leq k+1\\
 \C P^{k+1}~&\text{if}~j=k+2.
 \end{cases}
\]
\end{proposition}

\begin{proof}
If $1\leq j\leq k+1$, observe that $\lk_{L_{j,m}^k}(m)=L_{j-1,m-1}^{k-1}$ and $\rt_{L_{j,m}^k}(m)=\Delta_{\{1, \ldots, m-1\}}$. We have two homotopy pushouts of corresponding moment-angle complexes and their quotient spaces under the diagonal action
\[
\begin{tikzcd}
\mz_{L_{j-1,m-1}^{k-1}} \times S^1 
\arrow[r, "* \times \mathrm{id}"]
\arrow[d, "\mathrm{id} \times *"']
 & S^1\arrow[d]\\
\mz_{L_{j-1,m-1}^{k-1}}
\arrow[r]& \mz_{L_{j,m}^k}
\end{tikzcd}
\text{and}
\begin{tikzcd}
\mz_{L_{j-1,m-1}^{k-1}}
\arrow[r]\arrow[d, "q_{j-1,m-1}^{k-1}"]
 & *\arrow[d]\\
\mz_{L_{j-1,m-1}^{k-1}}/S_d^1\arrow[r]& \mz_{L_{j,m}^k}/S_d^1.
\end{tikzcd}
\]
Thus $\mz_{L_{j,m}^k}\simeq  S^1* \mz_{L_{j-1,m-1}^{k-1}}$
and $\mz_{L_{j,m}^k}/S^1_d\simeq C_{q^{k-1}_{j-1, m-1}}$.
Iterating  $\mz_{L_{j,m}^k}\simeq  S^1* \mz_{L_{j-1,m-1}^{k-1}}$, we obtain the homotopy equivalences $\mz_{L_{j,m}^k}\simeq S^{2j-1}* \mz_{\Delta_{m-j}^{k-j}}$ for $1\leq j\leq k+1$.

Next  consider that $\lk_{L_{j,m}^k}(1)=L_{j,m-1}^{k-1}$ and $\rt_{L_{j,m}^k}(1)=L_{j,m-1}^k$. Consider the homotopy pushouts of corresponding moment-angle complexes and their quotient spaces under the diagonal action
\[
\begin{tikzcd}
S^1  \times \mz_{L_{j,m-1}^{k-1}}
\arrow[r, "\mathrm{id} \times *"]
\arrow[d, " * \times \mathrm{id}"']
 & S^1 \times \mz_{L_{j,m-1}^k}\arrow[d, "f^k_{j,m}"]\\
\mz_{L_{j,m-1}^{k-1}}
\arrow[r]& \mz_{L_{j,m}^k}
\end{tikzcd}
\text{and}
\begin{tikzcd}
\mz_{L_{j,m-1}^{k-1}}
\arrow[r, "\simeq *"]\arrow[d, "q_{j,m-1}^{k-1}"]
 & \mz_{L_{j,m-1}^{k}}
 \arrow[d, "g_{j,m}^k"]\\
\mz_{L_{j,m-1}^{k-1}}/S_d^1
\arrow[r]
& \mz_{L_{j,m}^k}/S_d^1.
\end{tikzcd}
\]
By Lemma~\ref{nullJ}, the simplicial inclusion $\lk_{L_{j,m}^k}(1)\longrightarrow \rt_{L_{j,m}^k}(1)$ induces a null homotopic map on corresponding moment-angle complexes.
Thus, there are two splitting homotopy cofibrations
\[\begin{split}
S^1\ltimes \mz_{L_{j,m-1}^k} &\overset{\bar{f}^k_{j,m}}\longrightarrow \mz_{L_{j,m}^k} \longrightarrow S^1*\mz_{L_{j,m-1}^{k-1}}
 \\
\mz_{L_{j,m-1}^{k}} &\overset{g^k_{j,m}}\longrightarrow 
\mz_{L_{j,m}^{k}}/S_d^1 \longrightarrow C_{q_{j,m-1}^{k-1}}.
\end{split}
\]
Thus, there are homotopy equivalences 
\[
\begin{split}
\mz_{L_{j,m}^k}\simeq S^1*\mz_{L_{j,m-1}^{k-1}}\vee S^1\ltimes \mz_{L_{j,m-1}^k}
 & ~~\text{and}~ ~
 C_{\bar{f}^k_{j,m}}
  \simeq S^1*\mz_{L_{j,m-1}^{k-1}} 
  \simeq \mz_{L_{j+1,m}^{k}}
\\
\mz_{L_{j,m}^{k}}/S_d^1 \simeq \mz_{L_{j,m-1}^{k}} \vee C_{q_{j,m-1}^{k-1}}
& ~~\text{and}~~
 C_{g^k_{j,m}}
 \simeq C_{q_{j,m-1}^{k-1}}
  \simeq  \mz_{L_{j+1,m}^{k}}/S^1_d.
\end{split} 
\] 
Iterating the homotopy equivalence 
$\mz_{L_{j,m}^{k}}/S_d^1 
\simeq \mz_{L_{j,m-1}^{k}} \vee C_{q_{j,m-1}^{k-1}}
\simeq \mz_{L_{j,m-1}^{k}} \vee (\mz_{L_{j+1,m}^{k}}/S^1_d)$,
we have
\begin{equation}\label{zklq}
\mz_{L_{j,m}^{k}}/S_d^1 \simeq \mz_{L_{j,m-1}^{k}}\vee \mz_{L_{j+1,m-1}^{k}}\vee \ldots \vee \mz_{L_{k+1,m-1}^{k}}\vee (\mz_{L_{k+2,m}^{k}}/S^1_d).
\end{equation}

In the end, we identify the homotopy type of $\mz_{L_{k+2,m}^{k}}/S^1_d$.

If $k=0$, then $L_{2,m}^0=\Delta_{\{1,\ldots, m-1\}}\cup \Delta_{\{1,\ldots, m-2, m\}}$, where two $(m-2)$-simplices are glued together along one common facet $\Delta_{\{1, \ldots, m-2\}}$.
In this case, we have
  \[
  \mz_{L_{2,m}^0} = (\underset{m-2}\prod D^2) \times  (D^2, S^1)^{\partial \Delta^1} \simeq S^1*S^1.
  \] 
 Since the diagonal action on $\zk$ is free, the genuine quotient space has the same homotopy type as its homotopy quotient.
 Hence, there is a homotopy equivalence  
\[
\scalemath{0.9}{
\mz_{L_{2,m}^0}/S_d^1\simeq ES^1 \times_{S^1_d} \mz_{L_{2,m}^0} 
=ES^1 \times_{S_d^1} ((\underset{m-2}\prod D^2) \times  (D^2, S^1)^{\partial \Delta^1})\simeq 
ES^1 \times_{S_d^1} (D^2, S^1)^{\partial \Delta^1}
 \simeq  \C P^1}. 
 \] 
In general, the simplicial complex $L_{k+2, m}^{k}=\underset{j=m-k-1}{\overset{m}\bigcup}\Delta_{\{1,\ldots,\hat{j}, \ldots, m\}}$, where $k+2$ simplices of dimension $m-2$ (the ``first" $k+2$ facets of $\Delta^{m-1}$) are glued along the common face $\Delta_{\{1, \ldots, m-k-2\}}$.
Thus, $\mz_{L_{k+2,m}^{k}}=(\underset{m-k-2}\prod D^2) \times (D^2, S^1)^{\partial \Delta^{k+1}}$.
 The diagonal action on $\mz_{L_{k+2,m}^{k}}$ implies that the genuine quotient space has the same homotopy type with its homotopy quotient. Hence, we have
 \[
  \scalemath{0.85}{
 \mz_{L_{k+2,m}^{k}}/S_d^1\simeq ES^1\times_{S^1_d} \mz_{L_{k+2,m}^{k}}
  = ES^1 \times_{S^1_d} ((\underset{m-k-2}\prod D^2) \times (D^2, S^1)^{\partial \Delta^{k+1}}) 
   \simeq (D^2, S^1)^{\partial \Delta^{k+1}}/S^1_d \simeq \C P^{k+1}.
 }
 \]
By (\ref{zklq}), there is a homotopy equivalence
\[
\begin{split}
\mz_{L_{j,m}^{k}}/S_d^1 &\simeq  \C P^{k+1} \vee \mz_{L_{j,m-1}^{k}}\vee \mz_{L_{j+1,m-1}^{k}}\vee \ldots \vee \mz_{L_{k+1,m-1}^{k}} \\
&\simeq \C P^{k+1} \vee (S^{2j-1}* \mz_{\Delta_{m-j-1}^{k-j}}) \vee (S^{2j+1}*\mz_{\Delta_{m-j-2}^{k-j-1}})\vee \ldots \vee (S^{2k+1}*T^{m-k-2})\\
&\simeq \C P^{k+1} \vee (\underset{i=j}{\overset{k}\vee}S^{2i-1}*\mz_{\Delta_{m-i-1}^{k-i}})\vee (S^{2k+1}*T^{m-k-2}).
\end{split}
\]
\end{proof}

Now we  prove Theorem~\ref{ckm}.

\noindent \textit{Proof of Theorem~\ref{ckm}.}
The homotopy commutative diagram~(\ref{keycd}) shows that $C_{k,m}\simeq \mz_{L_{1, m+1}^{k+1}}/S^1_d$. 
 By Proposition~\ref{hmpjkm}, 
\[
C_{k,m} \simeq \C P^{k+2}\vee (\underset{i=1}{\overset{k+1}\vee}S^{2i-1}*\mz_{\Delta_{m-i}^{k+1-i}})\vee
 (S^{2k+3}*T^{m-k-2}).
\]
\qed

Together with Lemma~\ref{lemzkm}, we have the homotopy type of $\mz_{\Delta^k_{m}}/S_d^1$.
\begin{corollary}\label{zkkm}
The homotopy type of $\mz_{\Delta^k_{m}}/S_d^1$ is $\mz_{\Delta^k_{m-1}}\vee C_{k-1, m-1}$.
\end{corollary}\qed

\nocite{}

\bibliographystyle{plain}
\bibliography{nummeth}


\end{document}